\documentclass[10pt]{amsart}

\usepackage{amsmath}
\usepackage{amsfonts}
\usepackage{amssymb}
\usepackage[utf8]{inputenc}
\usepackage[T1]{fontenc}

\newtheorem{thrm}{Theorem}[section]

   \newtheorem{fact}[thrm]{Proposition}
   
   \newtheorem{col}[thrm]{Corollary}

   \newtheorem{problem}[thrm]{Problem}

\theoremstyle{remark}
\newtheorem{remark}[thrm]{Remark}

\title[A note on condensations of function spaces]{A note on condensations of function spaces onto $\sigma$-compact and analytic spaces}
\author{Miko\l aj Krupski}

\address{Institute of Mathematics\\ Polish Academy of Sciences\\ \newline Ul. \'Sniadeckich 8\\00--956 Warszawa\\ Poland }
\email{krupski@impan.pl}



\date{\today}
\thanks{The author was partially supported by the Polish National Science Center research grant UMO-2012/07/N/ST1/03525}

\begin{document}

\begin{abstract}
Modifying a construction of W.\ Marciszewski we prove (in {\textsf ZFC}) that there exists a subspace of the real line $\mathbb{R}$,
such that the realcompact space $C_p(X)$ of continuous real-valued functions on $X$ with the pointwise convergence topology does not admit a continuous
bijection onto a $\sigma$-compact space. This answers a question of Arhangel'skii.
\end{abstract}

\maketitle

\section{Introduction}
In this article, unless otherwise stated, by a space we mean a Hausdorff topological space.
For a space $X$, we denote by $C_p(X)$ the space of all continuous real-valued functions on $X$ equipped
with the pointwise convergence topology. It is well known that for a completely regular space $X$ the space $C_p(X)$ is almost
never $\sigma$-compact. Namely, it is $\sigma$-compact if and only if the space $X$ is finite (see \cite[1.2.186]{T}).
Thus it is natural to ask which $C_p(X)$ spaces admit a weaker $\sigma$-compact (compact) topology, approximating the original one.
This general problem posed
by Arhangel'skii can be reformulated using a concept of {\em condensation} i.e. continuous bijection:
{\em When there is a condensation of $C_p(X)$ onto a $\sigma$-compact (compact) space?} (see e.g. \cite[Problem 5.1]{M2}).

For example, as was proved by Michalewski in \cite{Mich}, the space $C_p(X)$ condenses onto a (metrizable) compactum
whenever $X$ is a metrizable analytic space i.e. $X$ is a continuous image of the space $\omega^\omega$ of the irrationals.
Let us recall two problems posed by Arhangel'skii concerning the possible generalizations of the result mentioned above.

\begin{problem}\label{problem1}(Arhangel'skii, \cite[Problem 4]{A})
Suppose, that $X$ is a separable metrizable space. Does $C_p(X)$ condense onto a $\sigma$-compact space?
\end{problem}

\begin{problem}\label{problem2}(Arhangel'skii, \cite[Problem 37]{A1})
Suppose, that $C_p(X)$ is realcompact. Does $C_p(X)$ condense onto a $\sigma$-compact space? 
\end{problem}

Note that Problem \ref{problem2}, which was also asked in the recent book of Tkachuk \cite[4.10.1]{T}, is more
general than Problem \ref{problem1}. Indeed, the space $C_p(X)$ is realcompact provided $X$ is separable (see \cite[Problems 418 and 429]{T}).

A consistent negative answer to Problem \ref{problem1} (and hence also to Problem \ref{problem2}) was given by Marciszewski in \cite{M1}.
He constructed, assuming that $\mathfrak{d}=2^\omega$, a subspace $X$ of the real line $\mathbb{R}$
such that $C_p(X)$ does not condense onto a $\sigma$-compact space. However the question if one can construct such a space without
any additional set-theoretic assumptions remained open
(see \cite[page 363]{M2}).

In this short note we show how to modify the construction from \cite{M1} to make it work in {\textsf ZFC}. Thus the answer to both aforementioned problems
is in the negative (see Theorem \ref{main} and Corollary \ref{col} below).

A significant part of the construction is the same
as in \cite{M1} and therefore some details will be omitted. The results presented in the main text of \cite{M1}
work for completely regular spaces. However, as was pointed out in \cite[Remark]{M1}, the construction
of a space $X$
can be modified to the effect that $C_p(X)$ cannot be condensed onto a $\sigma$-compact {\em Hausdorff} space.
It appears that as far as we are concerned with a condensation onto Tychonoff spaces, even the stronger result is true:
the space $C_p(X)$ cannot be condensed onto any Tychonoff analytic space.
We refer the interested reader to \cite[Chapter 5]{M2}, \cite{T1} and \cite{T} for more information on condensations in $C_p$-theory and
further references.

\section{Auxiliary results}
For a space $X$ and its subset $D$, we denote by $C_D(X)$ the subspace $\{f\upharpoonright D:f\in C_p(X)\}$ of
the space $C_p(D)\subseteq \mathbb{R}^D$. Note that the natural projection $\pi_D:\mathbb{R}^X\rightarrow \mathbb{R}^D$
condenses $C_p(X)$ onto $C_D(X)$, whenever $D$ is dense in $X$.

Proposition \ref{fakt1} given below was formulated in \cite{M1} without a proof (see \cite[Remark (2)]{M1}). Since it is
not completely straightforward we decided to enclose a short argument.

\begin{fact}\label{fakt1}
Suppose that a space $X$ is a countable union of metrizable compacta. Then there exists an injective Borel map of $X$ onto
a metrizable $\sigma$-compact space.
\end{fact}
\begin{proof}
Since a finite union of metrizable compacta is a metrizable compactum we have $X=\bigcup_{n\in \omega}K_n$, where $K_0=\emptyset$,
$K_n\subseteq K_{n+1}$ and each $K_n$ is metrizable compact.
For $n\in\omega$, let $i_n:K_n\rightarrow [0,1]^\omega$ be a homeomorphic embedding. Let $f_n:X\rightarrow [0,1]^\omega$ be a
Borel extension of $i_n$ taking $X\setminus K_n$ onto an arbitrary point.
The diagonal map $F=\bigtriangleup_{n\in \omega} f_n:X\rightarrow ([0,1]^\omega)^\omega$ is a Borel injection (every two points in $X$
belong to some $K_n$ and $f_n$ is injective on $K_n$). We shall prove that $F(X)$ is $\sigma$-compact.
To this end it suffices to show that for each $n\in \omega$ the map $F$ is continuous on the $\sigma$-compact
set $K_{n+1}\setminus K_n$. For $n=0$ this follows from the fact that each $f_n$ is continuous on $K_n$ and $K_1\subseteq K_n$, for $n\geqslant 1$.
Let $n\geqslant 1$. For $m\geqslant n+1$ the map $f_m$ is continuous on $K_{n+1}\setminus K_n\subseteq K_m$ and for $m<n+1$ the map $f_m$
is continuous on $K_{n+1}\setminus K_n$ being constant (we assumed that $f_m$ maps $X\setminus K_m$ onto a point).
\end{proof}

\begin{fact}\cite[Remark (3)]{M1}\label{factorization}
Let $E$ be a countable dense subset of a separable metrizable space $X$ and let $\psi:C_p(X)\rightarrow Y$ be continuous, where $Y$ is
an arbitrary second countable space.
Then $\psi$ has the following factorization property: There is a countable set $D\subseteq X$ containing $E$ and a
continuous map $\xi:C_D(X)\rightarrow Y$ such that $\psi=\xi\circ(\pi_D\upharpoonright C_p(X))$.
\end{fact}
\begin{proof}
Let $\{U_n\}_{n\in \omega}$ be a countable base for $Y$.
Since the space $X$ is separable metrizable it has a countable network. It follows that $C_p(X)$ also has a countable network and hence
it is hereditarily Lindel\"of. Thus, each open set in $C_p(X)$ is a countable union of basic open sets.

For each $n\in \omega$, we have $\psi^{-1}(U_n)=\bigcup_{m\in\omega}[A_m^n,V^n_m]$, where $[A_m^n,V_m^n]$ is a basic open set
in $C_p(X)$ given by a finite set $A_m^n\subseteq X$ and an open set $V_m^n\subseteq \mathbb{R}$.
It is easy to check that the set $D=E\cup\bigcup_{n,m\in\omega}A_m^n$ is as promised. The map $\xi$ is uniquely determined by $D$ because
density of $D$ implies that $\pi_D$ is injective on $C_p(X)$.
\end{proof}

The next proposition
can be proved by a simple transfinite induction (cf. \cite[the proof of 1.5.14]{vM}). It is a slight modification of the construction of the Bernstein set

\begin{fact}\label{Kuratowski}
There exists a set $\mathbb{K}\subseteq \mathbb{R}$ containing the rationals $\mathbb{Q}$
and such that both $\mathbb{K}$ and $\mathbb{R}\setminus \mathbb{K}$ intersect each copy of the Cantor set in $\mathbb{R}$.
\end{fact}
Since in every Cantor set we can find continuum many pairwise disjoint Cantor sets, both $\mathbb{K}$ and $\mathbb{R}\setminus \mathbb{K}$
have cardinality $2^\omega$.

The next, easy proposition will be crucial in our construction.

\begin{fact}\label{card}
Let $\mathbb{K}\subseteq \mathbb{R}$ be the set given by Proposition \ref{Kuratowski}. Suppose that $G\subseteq\mathbb{R}$ is a $G_\delta$ set
containing $\mathbb{K}$. Then the set $\mathbb{R}\setminus G$ is countable.
\end{fact}
\begin{proof}
Since $G$ is $G_\delta$, we have $\mathbb{R}\setminus G=\bigcup_{n\in\omega}F_n$, where each $F_n$ is closed.
If $\mathbb{R}\setminus G$ were uncountable one of the sets $F_n$,being closed and uncountable, would contain a copy of the Cantor set.
This however would contradict the property of $\mathbb{K}$.
\end{proof}

\section{The construction}

In this section we will prove the following
\begin{thrm}\label{main}
There is a space $X\subseteq\mathbb{R}$, containing the rationals~$\mathbb{Q}$, such that
$C_p(X)$ does not condense neither onto a $\sigma$-compact (Hausdorff) space, nor onto an analytic Tychonoff space.
 
\end{thrm}
\begin{proof}
Let $\mathcal{F}$ be the following family of maps (cf. \cite[Remark]{M1})
\begin{gather*}
\label{foo}\mathcal{F}=\{\varphi:B\rightarrow \mathbb{R}^\omega: \text{$B$ is an analytic subset of $\mathbb{R}^D$, for some
countable $D\subseteq \mathbb{R}$}\\ 
\label{bar}\text{ with $\mathbb{Q}\subseteq D$, $\varphi$ is Borel}\}.
\end{gather*}

The family $\mathcal{F}$ has cardinality $2^\omega$ and hence we can
enumerate it as $\{\varphi_\alpha:B_\alpha\rightarrow \mathbb{R}^\omega:\alpha<2^\omega\}$ (repetitions allowed) in such a way that
$B_\alpha$ is a subset of $\mathbb{R}^{D_\alpha}$.
Let $\mathbb{K}\subseteq \mathbb{R}$ be the set given by Proposition \ref{Kuratowski}.

In general, we are going to repeat the construction from \cite{M1}. The only change we are going to make is the starting point:
We will use the set $\mathbb{K}$ instead of $\mathbb{Q}$. This
simple idea allows us to drop the set-theoretic assumption $\mathfrak{d}=2^\omega$ required in \cite{M1} (cf. Remark \ref{remark} at the end of
the paper).
By induction we choose points $x_\alpha,y_\alpha\in \mathbb{R}\setminus \mathbb{K}$,
$G_\delta$-subsets $A_\alpha$ of $\mathbb{R}$ containing $\mathbb{K}$ and continuous functions $f_\alpha,g_\alpha:A_\alpha\rightarrow\mathbb{R}$
such that the following conditions are satisfied (we put $X_\alpha=\mathbb{K}\cup\{x_\beta:\beta<\alpha\}$).
\begin{itemize}
 \item[(i)] $x_\beta\neq x_\alpha$, for $\beta<\alpha$,
 \item[(ii)] $(X_\alpha\cup \{x_\alpha\})\cap\{y_\beta:\beta\leqslant\alpha\}=\emptyset$,
 \item[(iii)] $x_\alpha\in (\mathbb{R}\setminus \mathbb{K}) \cap(\bigcap_{\beta\leqslant\alpha}A_\beta)$,
 \item[(iv)] if $D_\alpha\setminus X_\alpha\neq\emptyset$, then $y_\alpha\in D_\alpha$,
 \item[(v)] if $D_\alpha\subseteq X_\alpha$ and $C_{D_\alpha}(X_\alpha)\setminus B_\alpha\neq\emptyset$,
then $X_\alpha\subseteq A_\alpha$ and $f_\alpha\upharpoonright D_\alpha\notin B_\alpha$,
 \item[(vi)] if $D_\alpha\subseteq X_\alpha$, $C_{D_\alpha}(X_\alpha)\subseteq B_\alpha$ and $\varphi_\alpha\upharpoonright C_{D_\alpha}(X_\alpha)$
is not injective, then $X_\alpha\subseteq A_\alpha$, $f_\alpha\neq g_\alpha$ and $\varphi(f_\alpha\upharpoonright D_\alpha)=\varphi(g_\alpha\upharpoonright D_\alpha)$.
\end{itemize}

Note that for $\alpha<2^\omega$ the set $(\mathbb{R}\setminus \mathbb{K})\cap (\bigcap_{\beta\leqslant\alpha}A_\beta)$ has cardinality $2^\omega$.
Indeed, for each $\beta\leqslant\alpha$ the set $A_\beta$ is a $G_\delta$-subset of $\mathbb{R}$ containing $\mathbb{K}$ and hence Proposition \ref{card}
implies that $|(\bigcap_{\beta\leqslant\alpha}A_\beta)^c|<2^\omega$. Moreover, as we have already observed, the set $\mathbb{R}\setminus \mathbb{K}$ has
cardinality $2^\omega$.

Now, the inductive step can be made by considering the same four cases (corresponding to conditions (iv)-(vi)) as in \cite{M1}.

Put $X=\mathbb{K}\cup \{x_\alpha:\alpha<2^\omega\}$. We need to show that there is no condensation from $C_p(X)$ onto a space $M$
being $\sigma$-compact or Tychonoff and analytic.
The argument given in \cite{M1} works for $M$ being Tychonoff and $\sigma$-compact. If $M$ is just Hausdorff the proof is a little bit different and for this
reason we enclose a justification.

Let $M$ be $\sigma$-compact. Suppose that there is a condensation $\psi:C_p(X)\rightarrow M$.
Since $X$ is separable metrizable, the space $C_p(X)$ has
a countable network and so does $M$. Thus $M$ is a countable union of metrizable compacta. Since every space with a countable
network condenses onto a space with a
countable base \cite[Ch.2, Problem 149]{AP}, without loss of generality we may assume that $M$ has a countable base. By Proposition \ref{factorization},
there exists a countable set $D\subseteq X$ containing $\mathbb{Q}$ and a map $\xi:C_D(X)\rightarrow M$ such that
$\psi=\xi\circ(\pi_D\upharpoonright C_p(X))$. Proposition \ref{fakt1} yields the existence of a metrizable $\sigma$-compact space $S$ and
an injective Borel map $\eta:M\rightarrow S$ (in the case $M$ is an analytic Tychonoff space, by \cite[Problem 156 (iii)]{T}
there exists a continuous map $\eta$ and a space $S$ with a countable base which is analytic as a continuous image of an analytic space $M$).
As $S$ is metrizable, we can assume that $S\subseteq \mathbb{R}^\omega$. We take
$\varphi'=\eta\circ\xi:C_D(X)\rightarrow \mathbb{R}^\omega$. Using a theorem of Kuratowski (see \cite[\S 35.VI]{K}) we can find
a Borel subset $B'$ of $\mathbb{R}^D$ containing $C_D(X)$ and a Borel extension $\varphi'':B'\rightarrow \mathbb{R}^\omega$ of a Borel injective
map $\varphi'$.
From now on we proceed as in \cite{M1}.
Take $B=(\varphi'')^{-1}(S)$ and $\varphi=\varphi''\upharpoonright B$. The set $B$ is analytic, being a Borel preimage of a $\sigma$-compact set
(if $M$ were a Tychonoff analytic space, the set $B$ is a Borel preimage of an analytic set hence it is analytic)
and $\varphi$ maps injectively $C_D(X)$ onto $S$. The map $\varphi$ belongs to $\mathcal{F}$ and thus, there is $\alpha<2^\omega$ such that
$\varphi=\varphi_\alpha$, $D=D_\alpha$, $B=B_\alpha$.

We consider the following four complementary cases:\\
{\bf Case 1:} $D_\alpha\setminus X_\alpha\neq\emptyset$.

By (iv), $y_\alpha\in D_\alpha=D$ and $y_\alpha\notin X$ by (ii). Hence $D\setminus X\neq\emptyset$, a contradiction.\\
{\bf Case 2:} $D_\alpha\subseteq X_\alpha$ and $C_{D_\alpha}(X_\alpha)\setminus B_\alpha\neq\emptyset$.

By (v), $X_\alpha\subseteq A_\alpha$ and by (iii) $x_\beta\in A_\alpha$ for $\beta\geqslant\alpha$, so $X\subseteq A_\alpha$ and
$f_\alpha\upharpoonright X\in C_p(X)$. Condition (v) implies that $f_\alpha\upharpoonright D\in C_D(X)\setminus B$, a contradiction.\\
{\bf Case 3:} $D_\alpha\subseteq X_\alpha$, $C_{D_\alpha}(X_\alpha)\subseteq B_\alpha$ and $\varphi_\alpha\upharpoonright C_{D_\alpha}(X_\alpha)$
is not injective.

Similarly as in Case 2, condition (vi) implies that $X\subseteq A_\alpha$ and
$f_\alpha\upharpoonright X, g_\alpha\upharpoonright X\in C_p(X)$. Then $f_\alpha\upharpoonright D, g_\alpha\upharpoonright D$ are distinct
elements of $C_D(X)$ and $\varphi(f_\alpha\upharpoonright D)=\varphi(g_\alpha\upharpoonright D)$, a contradiction.\\
{\bf Case 4:} $D_\alpha\subseteq X_\alpha$, $C_{D_\alpha}(X_\alpha)\subseteq B_\alpha$ and $\varphi_\alpha\upharpoonright C_{D_\alpha}(X_\alpha)$
is injective.

By (i), $X_\alpha$ is a proper subset of $X$ and hence $C_D(X)$ is a proper subset of $C_D(X_\alpha)$ (see \cite[Proposition 2.5]{M1}).
Since $\varphi$ is injective on $C_D(X_\alpha)$ and $\varphi(C_D(X_\alpha))\subseteq\varphi(B)=S$, we have $S\setminus \varphi(C_D(X))\neq\emptyset$,
a contradiction.
\end{proof}
Theorem \ref{main} gives a negative answer to Problem \ref{problem1}. As we mentioned in Introduction it also immediately implies the following negative
answer to Problem \ref{problem2}

\begin{col}\label{col}
There is a space $C_p(X)$ which is realcompact and does not admit a condensation onto a $\sigma$-compact space. 
\end{col}
\begin{remark}
It seems that the answer to Problem \ref{problem2} has been known before. As R. Pol observed
it is enough to consider $X=\omega\cup\{\infty\}$, where the points of $\omega$ are isolated and the neighborhoods of $\infty$
are given by an analytic non-Borel filter on $\omega$ (see \cite[Theorem 4.1]{LvMP}). Then $C_p(X)$ is clearly realcompact and does not condense onto
any Borel set.
\end{remark}

\begin{remark}\label{remark}
We feel that the reader deserves a better explanation why the proof of Theorem \ref{main}, although very similar to the construction from \cite{M1},
does not require the set-theoretic assumption $\mathfrak{d}=2^\omega$ which was vital in \cite{M1}.
Recall that $\mathfrak{d}=2^\omega$ is equivalent to the following statement:
\begin{gather*}
\label{up}(\ast) \text{ The intersection of less than continuum many } G_\delta\text{-subsets of }\mathbb{R}\text{ containing }\mathbb{Q}\\
\label{down}\text{ has cardinality }2^\omega.
\end{gather*}

Now, if we start the construction of the space $X$ from $\mathbb{Q}$ and try to add points $x_\alpha\in\mathbb{R}\setminus\mathbb{Q}$ inductively,
as it was done in \cite{M1}, we will need $(\ast)$ to fulfill conditions (i)-(iii) of the construction. That is, we will need $(\ast)$
to make sure that the intersection
$(\mathbb{R}\setminus\mathbb{Q})\cap\bigcap_{\beta\leqslant\alpha}A_\beta$ is big enough (has cardinality $2^\omega$) to choose a point
$x_\alpha$ belonging to it and distinct from points already chosen.

If we start the construction of $X$ from $\mathbb{K}$ (as it was done in the proof of Theorem \ref{main}) we force the
intersection $(\mathbb{R}\setminus\mathbb{K})\cap\bigcap_{\beta\leqslant\alpha}A_\beta$ to be big, simply by making sure that
it contains the set $\mathbb{K}$. Thus we no longer need the statement $(\ast)$.
\end{remark}

\textbf{Acknowledgment.}
\medskip

I would like to thank Witold Marciszewski for discussions on the subject. I am
also indebted to Roman Pol for valuable suggestions and to Pawe{\l} Krupski and Grzegorz Plebanek for reading the
preliminary version of the paper.

\end{document}